\documentclass{amsart}

\usepackage{enumerate}
\usepackage{mathtools}
\usepackage{tikz-cd}

\newtheorem{thm}{Theorem}[section]
\newtheorem{lemma}[thm]{Lemma}
\newtheorem{thmIntro}{Theorem}

\theoremstyle{definition}
\newtheorem{defn}[thm]{Definition}

\theoremstyle{remark}

\numberwithin{equation}{section}

\newcommand{\dwst}[0]{_{*}}

\newcommand{\mC}{\mathbb{C}}
\newcommand{\mH}{\mathbb{H}}
\newcommand{\mD}{\mathbb{D}}
\newcommand{\mS}{\mathbb{S}}

\newcommand{\mZ}{\mathbb{Z}}

\newcommand{\myeqref}[1]{\text{(}\ref{#1}\text{)}}

\begin{document}

\title{Some Remarks on Isoparametric Functions in Closed 4-Manifolds}

\author{Minghao Li}
\address{Shanghai Center for Mathematical Sciences, Fudan University, Shanghai, 200438, China}
\email{mhli19@fudan.edu.cn}






\begin{abstract}
  We study transnormal and isoparametric functions on closed Riemannian 4-manifolds and establish fundamental restrictions on their topology and geometry. In particular, we show that such manifolds cannot be endowed with negatively curved metrics, contrasting with known results in the compact simply connected case. Moreover, in certain cases, we provide a description of their fundamental groups. These findings contribute to a better understanding of the global structure of isoparametric foliations.
\end{abstract}

\maketitle

\section{Introduction}

    A smooth function $f$ on a Riemannian manifold $(M,g)$ is called {\it transnormal} if there is a smooth function $b$ such that $|\nabla f|^2=b(f)$. If, in addition, $\Delta f=a(f)$ for another smooth function $a$, then $f$ is called {\it isoparametric}. Each regular level hypersurface of an isoparametric function is called an {\it isoparametric hypersurface}. The classification of isoparametric hypersurfaces in certain special spaces has a long and rich history. The study of isoparametric hypersurfaces was significantly advanced by Cartan (\cite{Cartan1,Cartan2}), who completed the classification in Euclidean and hyperbolic spaces and initiated the investigation in spheres. The discovery of inhomogeneous examples (\cite{ozeki1975some,ferus1981cliffordalgebren}) highlighted the complexity of the spherical case, leading to extensive research over several decades and culminating in a complete classification (\cite{takagi1972principal,munzner1980isoparametric,munzner1981isoparametric,abresch1983isoparametric,dorfmeister1985isoparametric,cecil2007isoparametric,immervoll2008classification,chi2011isoparametric,chi2013isoparametric,miyaoka2013isoparametric,miyaoka2016errata,chi2020isoparametric}). Similar classification problems have also been studied in other spaces, including complex projective spaces $\mC P^n$ (\cite{wang1982isoparametric,kimura1986real,park1989isoparametric,dominguez2016isoparametric}), complex hyperbolic spaces $\mC\mH^n$ (\cite{berndt1989real,berndt2006real,berndt2007real,diaz2012inhomogeneous,diaz2017isoparametric}), compact symmetric spaces (\cite{murphy2012curvature}), and the product space $\mS^2\times\mS^2$ (\cite{urbano2019hypersurfaces}).
    \par    
    For general Riemannian manifolds, a significant result is Wang's work \cite{isoparaWang1987} on the tubular regularity of transnormal functions on connected complete Riemannian manifolds, showing that the focal varieties, which are the level sets taking the maximal or minimal values, are submanifolds, while the other level sets are tubes over either of the focal varieties. This result directly leads to the fact that an arbitrary transnormal function can induce an embedded {\it transnormal system} (see \cite{bolton1973transnormal}) of codimension 1, which also forms a singular Riemannian foliation of codimension one. 
    \par   
    Building on these developments, subsequent studies have investigated isoparametric functions and isoparametric foliations under various specific conditions (e.g., \cite{ge2015differentiable,qian2016differential,ge2016isoparametric,peng2017ricci}), with most focusing on particular cases. More recently, in collaboration with Yang, the author established some equivalent conditions for the existence of isoparametric or transnormal functions on closed Riemannian manifolds, stated as follows.
    \begin{thmIntro}[\cite{li2025equiv}]\label{thm: isopara and LDDBD}
        For each closed smooth manifold $M$, the following statements are equivalent:
        \begin{enumerate}[(a)]
            \item $M$ admits a linear double disk bundle decomposition.
            \item $M$ can be endowed with a Riemannian metric so that it admits an embedded transnormal system of codimension 1.
            \item $M$ can be endowed with a Riemannian metric so that it admits a transnormal function.
            \item $M$ can be endowed with a Riemannian metric so that it admits an isoparametric function. 
        \end{enumerate}
    \end{thmIntro}
    \par   
    Here, a linear double disk bundle decomposition means that the manifold $M$ can be expressed as the union of two linear disk bundles, glued along their common boundary by a diffeomorphism. The term "linear" refers to these disk bundles being subbundles of vector bundles, with fibers being unit disks. Notably, this includes the special case of one-dimensional disks. 
    \par    
    For closed simply connected 4-manifolds, \cite{ge2015differentiable} has classified singular Riemannian foliations of codimension 1 by employing the double disk bundle decomposition. This classification extends to isoparametric foliations. Moreover, the manifolds admitting such structures can be explicitly characterized, as stated in the following theorem:
    \begin{thmIntro}[\cite{ge2015differentiable}] \label{thm: Ge elliptic in 3d}
        Let $M$ be a closed simply connected Riemannian 4-manifold admitting an isoparametric function. Then $M$ is diffeomorphic to one of the following: the standard $\mS^4$, $\mC\mathbb{P}^2$, $\mS^2\times\mS^2$, or $\mC\mathbb{P}^2\#\pm \mC\mathbb{P}^2$. 
    \end{thmIntro}
    \par    
    This result was later extended to dimension five in \cite{devito2023manifolds}, where the possible diffeomorphism types of such manifolds are $\mS^5$, the Wu manifold $\operatorname{SU}(3)/\operatorname{SO}(3)$, $\mS^3\times\mS^2$, and the unique non-trivial $\mS^3$-bundle over $\mS^2$.
    \par    
    In contrast, results concerning closed manifolds with general connectivity remain sparse. In this paper, we focus on closed Riemannian 4-manifolds admitting a transnormal (or isoparametric) function without assuming simple connectivity, and establish the following theorem:
    \begin{thmIntro}\label{thm: exclu of hyperbolic}
        Let $M$ be a closed Riemannian 4-manifold admitting an isoparametric function. Then, the underlying manifold $M$ cannot be equipped with a negatively curved metric. 
    \end{thmIntro}
    \par  
    This theorem contrasts with the ellipticity in Theorem \ref{thm: Ge elliptic in 3d}, and it is nontrivial as the conclusion does not hold in dimension three (see Section \ref{sec: fail in 3d}). By imposing a further restriction, we obtain a more refined result: 
    \begin{thmIntro}\label{thm: fundamental group when singular}
        Let $M$ be a closed Riemannian 4-manifold admitting an isoparametric function $f$.
        \begin{enumerate}[(I)]
            \item If two connected components of the level sets of $f$ have codimension greater than 1, then one of the following occurs:
            \begin{enumerate}[(i)]
                \item the fundamental group of $M$ is abelian;
                \item the fundamental group of $M$ is isomorphic to a semidirect product of two cyclic groups;
                \item $M$ is a 2-sphere bundle over a closed hyperbolic surface.
            \end{enumerate}
            \item If only one connected component of the level sets of $f$ has codimension greater than 1, then $M$ has a connected double cover, which satisfies one of the conclusions in (I). 
        \end{enumerate}
    \end{thmIntro}
    \par    
    Theorems \ref{thm: exclu of hyperbolic} and \ref{thm: fundamental group when singular} establish fundamental constraints on the geometric properties of isoparametric foliations. Isoparametric foliations can be viewed as a generalization of cohomogeneity-one actions, which are typically associated with non-negative curvature (see, for example, \cite{neumann19683,Parker19864dimensionalGW,grove2002cohomogeneity,hoelscher2010classification}). However, much less is known about their curvature behavior. The results above indicate that, while isoparametric foliations extend beyond cohomogeneity-one actions, their curvature behavior remains subject to significant restrictions. These findings contribute to a deeper understanding of the interplay between isoparametric structures and the global geometric properties of the underlying manifold.

\section{Preliminaries}\label{sec: preliminaries}
    This section introduces the necessary background and foundational concepts for the subsequent discussions. It also includes key notations and results that will be referenced throughout the paper. The subsection on homotopy groups is largely based on \cite{hatcher2002algebraic}. 
    
\subsection{Homotopy groups.}
    \begin{defn}
        Let $\mathbf{I}^n$ be the $n$-dimensional unit cube. For a topological space $X$ with a fixed point $x_0\in X$, the \textit{n-th homotopy group} of $X$, denoted by $\pi_n(X,x_0)$, is the set of all continuous maps $f:\mathbf{I}^{n}\to X$ such that $f\left(\partial\, \mathbf{I}^{n}\right)=x_0$, under the equivalence relation of homotopies that preserves $\partial\, \mathbf{I}^{n}$. For $f_1,f_2\in\pi_n(X,x_0)$, define $f_1f_2$ by 
        \[
            (f_1f_2)(t_1,...,t_n)=\left\{
                \begin{aligned}
                    &f_1(2t_1,t_2,...,t_n) & \text{for}\ 0\le t_1\le \frac12,\\
                    &f_2(2t_1-1,t_2,...,t_n) & \text{for}\ \frac12\le t_1\le 1.
                \end{aligned}
            \right.
        \]
        The first homotopy group is also called the \textit{fundamental group}, and others are the \textit{higher homotopy groups}.
    \end{defn}
    \par 
    This definition can be extended to the case $n=0$ by letting $\mathbf{I}^{0}$ be a single point and $\partial\, \mathbf{I}^{0}$ the empty set. In this case, $\pi_0(X)$ corresponds to the set of path-components of $X$. Furthermore, if $\pi_0(X)$ is trivial, meaning that $X$ is path-connected, the homotopy groups of $X$ can be denoted simply by $\pi_n(X)$. If $X$ is not path-connected, the homotopy groups are typically considered separately for each path-component of $X$.
    \par 
    For the fundamental group, certain topological structures on a topological space can offer additional insights into its behavior. Assume a topological space $X$ is decomposed as the union of two path-connected open subsets $A_{1}$ and $A_{2}$, the intersection of which is also path-connected. Then, the inclusions $A_{k}\hookrightarrow X$ induce the homomorphisms from $\pi_1(A_k)$ to $\pi_1(X)$, and they extend to a homomorphism $\Phi$ from the free product $\pi_1(A_1)*\pi_1(A_2)$ to $\pi_1(X)$. Meanwhile, the inclusions $i_{k}:A_{1}\cap A_2\to A_k$ induce the homomorphisms ${i_{k}}_{*}: \pi_1(A_1\cap A_2)\to \pi_1(A_k)$. Under these notations, the well-known Seifert-van Kampen theorem is stated as follows:
    \par
    \begin{lemma}[Seifert-van Kampen Theorem]\label{lem: Seifert van kampen thm}
        Suppose $X=A_1\cup A_2$ with $A_{1}$, $A_{2}$ and $A_1\cap A_2$ all being path-connected open subset of $X$. Then the homomorphism $\Phi:\pi_1(A_1)*\pi_1(A_2)\to \pi_1(X)$ is surjective, and the kernel of $\Phi$ is the normal subgroup $N$ generated by all the elements of form ${i_{1}}_{*}(\omega)\,{i_{2}}_{*}(\omega)^{-1}$ for $\omega\in \pi_1(A_1\cap A_2)$. Hence, $\pi_1(X)\cong \pi_1(A_1)*\pi_1(A_2)\big/N$.
    \end{lemma}
    \par  
    For negatively curved manifold $M$, i.e., a smooth manifold which can be endowed with a Riemannian metric of negative sectional curvatures, the Cartan-Hadamard theorem implies that $M$ has trivial higher homotopy groups. Moreover, the fundamental group of such a manifold exhibits remarkable properties.
    \begin{lemma}[{\cite[Chap.12]{do1992-riemannian}}] \label{lemma: hyper non abelian}
        Let $M$ be a closed Riemannian manifold with negative sectional curvatures. $\pi_1(M)$ is not abelian.
    \end{lemma}
    \begin{lemma}[Preissmann's theorem]\label{lemma: hyper sub is Z}
        Let $M$ be a compact Riemannian manifold with negative sectional curvatures. Then any non-trivial abelian subgroup of the fundamental group $\pi_1(M)$ is isomorphic to $\mZ$.
    \end{lemma}
    \par    
    Next, we introduce the concept of a fiber bundle and establish a refined property of homotopy groups.
    \begin{defn}
        A sequence $F\xhookrightarrow{i} E\xrightarrow{p} B$ is called a \textit{fiber bundle} if 
        \begin{enumerate}[(i)]
            \item $E$, $B$ and $F$ are topological spaces, called the \textit{total space}, \textit{base space} and \textit{fiber} respectively;
            \item $p:E\to B$ is a continuous surjective map, called the \textit{fiber bundle projection} or simply \textit{projection};
            \item the inverse image $F_b=p^{-1}(b)$ is homeomorphic to $F$ for each point $b\in B$, called the \textit{fiber over $b$};
            \item $B$ has an open covering $\{U_{\alpha}\}_{\alpha\in\Lambda}$ such that for each $\alpha\in\Lambda$, there is a homeomorphism $\psi_{\alpha}:U_{\alpha}\times F\to p^{-1}(U_{\alpha})$ such that the composite $p\circ\psi_{\alpha}$ is the projection to the first factor.
        \end{enumerate}
    \end{defn}
    \par  
    A fiber bundle induces a long exact sequence of homotopy groups, which is a standard result in algebraic topology (see, for example, \cite[\S 4.2]{hatcher2002algebraic}).
    \begin{lemma}\label{lemma: long sequ of fibration}
        Let $B$ be a smooth connected manifold, and let $F\xhookrightarrow{i} E\xrightarrow{p} B$ be a fiber bundle. Choose an arbitrary base point $b_0\in B$ and $x_0\in F= p^{-1}(b_0)$. Then, there exists a long exact sequence as follows:
        \begin{align*}
            \cdots\rightarrow\pi_n(F,x_0)&\xrightarrow{i_{*}}\pi_n(E,x_0)\xrightarrow{p_{*}}\pi_n(B,b_0)\rightarrow\pi_{n-1}(F,x_0)\xrightarrow{i_{*}} \cdots
            \\
            \cdots&\xrightarrow{i_{*}}\pi_0(E,x_0)\xrightarrow{p_{*}} \pi_0(B,x_0)=1.
        \end{align*}
    \end{lemma}
    \par

\subsection{Double disk bundle decomposition}
    Let $M$ be a smooth manifold of the following structure:
    \begin{itemize}
        \item $M=\mathcal{D}_{+} \bigcup \mathcal{D}_{-}$, where $\mathcal{D}_{+}$ and $\mathcal{D}_{-}$ are closed subsets of $M$ with $\mathcal{D}_{+}\bigcap\mathcal{D}_{-}=\partial\mathcal{D}_{+}=\partial\mathcal{D}_{-}$;
        \item $\mathcal{P}_{+}:\mathcal{D}_{+}\to L_{+}$ (resp. $\mathcal{P}_{-}:\mathcal{D}_{-}\to L_{-}$) is a closed $k_{+}$-disk (resp. $k_{-}$-disk) fiber bundle over an embedded closed submanifold $L_{+}$ (resp. $L_{-}$) of $\mathcal{D}_{+}$ (resp. $\mathcal{D}_{-}$), where $k_{\pm}\ge 1$;
        \item the attaching map between $\mathcal{D}_{+}$ and $\mathcal{D}_{-}$ is a diffeomorphism $\phi:\partial\mathcal{D}_{+}\to\partial\mathcal{D}_{-}$.
    \end{itemize}
    We say $(\mathcal{P}_{\pm}:\mathcal{D}_{\pm}\to L_{\pm},\phi:\partial\mathcal{D}_{+}\to\partial\mathcal{D}_{-})$ is a double disk bundle decomposition of the smooth manifold $M$. Moreover, if each disk bundle is linear, that is, it is a disk subbundle of a vector bundle, we say $(\mathcal{P}_{\pm}:\mathcal{D}_{\pm}\to L_{\pm},\phi:\partial\mathcal{D}_{+}\to\partial\mathcal{D}_{-})$ is a {\bf linear double disk bundle decomposition}({\bf LDDBD}) of $M$. 
    \par    
    \par    
    In this paper, we mainly restrict our discussion to the case $\dim M=4$. In this setting, if $M$ admits a double disk bundle decomposition, then for each disk bundle, it is either a $4$-disk (over a single point) or a $k$-disk bundle over a closed manifold, where $k\le 3$. Since the orientation-preserving diffeomorphism group of a $k$-disk deformation retracts to $\operatorname{SO}(k)$ for $k\le 3$, these disk bundles are linear disk bundles in the sense of diffeomorphism. Thus, we can always assume that the double disk bundle decomposition of a 4-manifold is linear.

\section{Proofs of the Main Theorems}
    In this section, we first establish a lemma describing the behavior of fundamental groups under double disk bundle decompositions. Using this result, we then proceed to prove the main theorems.
\subsection{Auxiliary results on the fundamental group}
    We investigate the fundamental groups of manifolds admitting a double disk bundle decomposition. To begin, we introduce an algebraic lemma, which is then applied to a particular case of the Seifert–van Kampen theorem.
    \begin{lemma}\label{lemma: algebraic lemma}
        Let $G,U,V$ be groups and $u:G\to U$, $v:G\to V$ be epimorphisms. Denote $N$ the normal subgroup of $U*V$ generated by all the elements of the form $u(g)v(g^{-1})$ for $g\in G$. Then, $U*V\big/N$ is isomorphic to $U\big/u(\operatorname{Ker}v)$. 
    \end{lemma}
    \begin{proof}
        Denote $K=u(\operatorname{Ker}v)$ in short, and $e_U,e_V,e_G$ the identity elements of $U,V,G$, respectively. Clearly, $\phi(x\,K)=x\,N$ defines a homomorphism from $U\big/K$ to $U*V\big/N$. Next, we construct a map $\psi:U*V\big/N\to U\big/K$ as follows. For any element $z\in U*V$, it is an alternating product of elements of $u(G)$ and $v(G)$ since $u,v$ are surjective. Therefore, $z$ can be written as $u(g_1)v(g_2)...v(g_{l})$. Then, we define 
        \[
            \psi(z)=\psi\big(u(g_1)v(g_2)...v(g_l)\,N\big):=u(g_1g_2...g_l)\,K.
        \]
        In order to check that $\psi$ is well-defined, we will verify 
        \begin{enumerate}[(a)]
            \item $\psi(n\,N)=e_U\,K$ for any $n\in N$;
            \item $u(g_1g_2...g_l)\,K=u(\bar{g}_1\bar{g}_2...\bar{g}_l)\,K$ if $u(g_1)=u(\bar{g}_1)$, $v(g_2)=v(\bar{g}_2)$ and so on. 
        \end{enumerate}
        \par  
        Note that $N$ is a subgroup generated by all the conjugates of the elements $u(g)v(g^{-1})$. 
        For an element $n=zu(g)v(g^{-1})z^{-1}\in N$, after writing $z$ as the form before, we directly have $\psi(n\, N)=e_U\,K$, which ensures (a). If $u(g_1)=u(\bar{g}_1)$, $v(g_2)=v(\bar{g}_2)$ and so on, we claim $\bar{g}_1\bar{g}_2...\bar{g}_l=g_1g_2...g_lh$ for some $h\in \operatorname{Ker}(u)\operatorname{Ker}(v)$. It is proved by induction on $l$. Assume $\bar{g}_1\bar{g}_2...\bar{g}_k=g_1g_2...g_kh$ with $h\in \operatorname{Ker}(u)\operatorname{Ker}(v)$, and we consider the case $l=k+1$. We have $\bar{g}_{k+1}=g_{k+1}h_{k+1}$ for some $h_{k+1}\in \operatorname{Ker}(u)\operatorname{Ker}(v)$, and 
        \[
            h\bar{g}_{k+1}=hg_{k+1}h_{k+1}=g_{k+1}\left(g_{k+1}^{-1}hg_{k+1}\right)h_{k+1}.
        \]
        Notice that $\bar{h}:=\left(g_{k+1}^{-1}hg_{k+1}\right)h_{k+1}$ still lies in $\operatorname{Ker}(u)\operatorname{Ker}(v)$, so we have obtained that $\bar{g}_1\bar{g}_2...\bar{g}_{k+1}=g_1g_2...g_{k+1}\bar{h}$, which proves the claim. Hence, condition (b) follows. Moreover, a similar process to proving (b) implies that $\psi$ is a homomorphism. 
        \par 
        Obviously, $\psi\circ\phi$ is the identity on $U\big/K$. As for the converse, $\phi\circ\psi(u(g)\,N)=u(g)\,N$ and $\phi\circ\psi(v(g)\,N)=u(g)\,N=u(g)\big(u(g^{-1})v(g)\big)\,N=v(g)\,N$, which together means $\phi\circ\psi$ is also the identity. Therefore, $U\big/K$ is isomorphic to $U*V\big/N$.
    \end{proof}
    \par 
    Lemma \ref{lemma: algebraic lemma} enables us to establish the following theorem regarding the fundamental group of manifolds admitting a double disk bundle decomposition.
    \par 
    \begin{thm}\label{thm: fund group of M}
        Let $M$ be a smooth manifold admitting a double disk bundle decomposition $(\mathcal{P}_{\pm}:\mathcal{D}_{\pm}\to L_{\pm},\phi:\partial\mathcal{D}_{+}\to\partial\mathcal{D}_{-})$, and assume the codimensions of $L_{\pm}\subset M$ are both higher than one. Denote $i_{\pm}$ the inclusions $\partial\mathcal{D}_{\pm}\xhookrightarrow{}\mathcal{D}_{\pm}$. Then, the fundamental group $\pi_1(M)$ is isomorphic to a quotient group $\pi_1(L_{+})\big/N$, where $N$ is isomorphic to ${i_{+}}_{*}\big(\operatorname{Ker}({i_{-}}_{*}\circ\phi)\big)$ with ${i_{\pm}}_{*}$ the corresponding induced homomorphisms between the fundamental groups. Moreover, if the codimension of $L_{-}$ is higher than two, then $\pi_1(M)$ is isomorphic to $\pi_1(L_{+})$. 
    \end{thm}
    \begin{proof}
        Assume $d_{\pm}=\dim L_{\pm}$. Since $n-d_{\pm}\ge 1$, the fiber $\mD^{n-d_{\pm}}$ of both disk bundles are contractible, which implies that ${\mathcal{P}_{{\pm}}}_{*}$ is a natural isomorphism between $\pi_1(\mathcal{D}_{\pm})$ and $\pi_1(L_{\pm})$. Each boundary $\partial\mathcal{D}_{\pm}$ is an $\mS^{n-d_{\pm}-1}$-bundle over $L_{\pm}$, that is, $\mS^{n-d_{\pm}-1}\xhookrightarrow{}\partial\mathcal{D}_{\pm}\xrightarrow{p_{\pm}} L_{\pm}$. The long exact sequence on homotopy groups is 
        \begin{equation}\label{equ: long sequence on DL1 joint DL2}
            \cdots\rightarrow\pi_2(L_{\pm})\rightarrow\pi_1(\mS^{n-d_{\pm}-1})\rightarrow\pi_1(\partial \mathcal{D}_{\pm})\xrightarrow{{p_{{\pm}}}_{*}}\pi_1(L_{\pm})\rightarrow 1.
        \end{equation}
        Therefore, ${p_{{\pm}}}_{*}:\pi_1(\partial \mathcal{D}_{\pm})\rightarrow\pi_1(L_{\pm})$ are surjective. Consider the following commutative diagram and the induced commutative diagram among their fundamental groups:
        \begin{equation*}
            \begin{tikzcd}
                \partial \mathcal{D}_{\pm} \arrow{rd}{p_{{\pm}}} \arrow{d}{i_{{\pm}}} &  \\
                \mathcal{D}_{\pm} \arrow{r}{\mathcal{P}_{{\pm}}} & L_{\pm}.
            \end{tikzcd}
        \end{equation*}
        It is then known that ${i_{{\pm}}}_{*}:\pi_1\big(\partial\mathcal{D}_{\pm}\big)\to \pi_1(\mathcal{D}_{\pm})$ are epimorphisms. Applying Seifert-van Kampen theorem (Lemma \ref{lem: Seifert van kampen thm}) on $\mathcal{D}_{+}\cup\mathcal{D}_{-}$ and taking Lemma \ref{lemma: algebraic lemma} into consideration, we obtain that $\pi_1(M)$ is isomorphic to $\pi_1(L_{+})\big/{\mathcal{P}_{+}}\dwst\big({i_{+}}_{*}\big(\operatorname{Ker}({i_{-}}_{*}\circ\phi)\big)\big)$. 
        \par  
        If $n-d_{-}\ge 2$, then $\pi_1(\mS^{n-d_{-}-1})$ is trivial. The long exact sequence \myeqref{equ: long sequence on DL1 joint DL2} implies that ${p_{-}}_{*}:\pi_1(\partial \mathcal{D}_{-})\rightarrow\pi_1(L_{-})$ is an isomorphism, and so is ${i_{-}}_{*}:\pi_1\big(\partial\mathcal{D}_{-}\big)\to \pi_1(\mathcal{D}_{-})$. Hence, $\operatorname{Ker}({i_{-}}_{*}\circ\phi)$ is trivial and the conclusion follows.
    \end{proof}

\subsection{Discussion of double disk bundle decompositions}\label{subsec: in 4D}
    Assume $M$ is a closed Riemannian 4-manifold admitting an isoparametric function.  Recall the discussion in \cite{li2025equiv}, or briefly, Theorem \ref{thm: isopara and LDDBD}. The underlying manifold $M$ admits an LDDBD, where the connected components of the level sets of $f$ precisely correspond to the fiber subbundles whose fibers are hyperspheres and the base spaces of the two disk bundles. Denote $(\mathcal{P}_{\pm}:\mathcal{D}_{\pm}\to L_{\pm},\phi:\partial\mathcal{D}_{+}\to\partial\mathcal{D}_{-})$ the LDDBD. Use the same notation as above: $d_{\pm}=\dim L_{\pm}$, $i_{\pm}$ the inclusions $\partial\mathcal{D}_{\pm}\xhookrightarrow{}\mathcal{D}_{\pm}$, and $p_{\pm}=\mathcal{P}_{\pm}\circ i_{\pm}$ the natural projections $\partial\mathcal{D}_{\pm}\to L_{\pm}$. We will then discuss on $d_{\pm}$ to prove our main theorems. 
    \paragraph{\bf Case 1: $d_{\pm}\le 2$.} We now discuss the values of $d_{+}$ and $d_{-}$ in further detail.
    \begin{enumerate}[1)]
        \item $d_{+}<2$ (or $d_{-}<2$). In this case, $L_{+}$ is a closed $d_{+}$-manifold, which is homeomorphic to either a single point or a circle. The fundamental group $\pi_1(L_{+})$ is isomorphic to either $\{1\}$ or $\mZ$, and then it follows from Theorem \ref{thm: fund group of M} that $\pi_1(M)$ is either trivial or isomorphic to a quotient group of $\mZ$. Anyway, $\pi_1(M)$ is abelian, and it will contradict to Lemma \ref{lemma: hyper non abelian} if $M$ can be equipped with a negatively curved metric. 
        \item $d_{+}=d_{-}=2$. Both $L_{+}$ and $L_{-}$ are closed surfaces. 
        \begin{itemize} 
            \item If $L_{+}$ (or $L_{-}$) is homeomorphic to the sphere, the real projective plane, or the torus, a similar discussion as in the case $d_{+}<2$ remains valid. 
            \item If $L_{+}$ (or $L_{-}$) is homeomorphic to a closed hyperbolic surface, then $\partial\mathcal{D}_{+}$ is a Seifert fibered space with a hyperbolic base. Since $\partial\mathcal{D}_{-}$ is also a Seifert fibered space, the homeomorphism $\phi:\partial \mathcal{D}_{+}\to \partial \mathcal{D}_{-}$ is homotopically fiber-preserving by the theory of Seifert fibered spaces. It follows from Waldhausen's theorem (\cite{waldhausen1968irreducible}) that $\phi$ is isotopically fiber-preserving. Consequently, up to isotopy, $\phi$ induces an isomorphism between fiber bundles $\mathcal{D}_{+}\to L_{+}$ and $\mathcal{D}_{-}\to L_{-}$, and $M$ is an $\mS^2$-bundle over $L_{+}$. From the long exact sequence on homotopy groups associated with $\mS^2\xhookrightarrow{}M\xrightarrow{}L_{+}$, it follows that $\pi_1(M)$ is isomorphic to $\pi_1(L_{+})$. If $M$ can be equipped with a negatively curved metric, then the projection $p:M\to L_{+}$ would be a homotopy equivalence by the Whitehead theorem. However, this is impossible, as closed negatively curved manifolds of different dimensions cannot be homotopy equivalent. 
            \item If both $L_{+}$ and $L_{-}$ are homeomorphic to the Klein bottle $\mathbf{K}$, we have the following commutative diagram:
            \begin{equation*}
                \begin{tikzcd}
                    \pi_2(\mathbf{K})=1\arrow{r} & \pi_1(\mS^1)\cong\mZ \arrow{r} \arrow{d}{} & \pi_1\big(\partial \mathcal{D}_{\pm}\big) \arrow{r}{{p_{\pm}}_{*}} \arrow{d}{{i_{\pm}}_{*}} & \pi_1(\mathbf{K}) \arrow{r}\arrow{d}{=} & 1 \\
                    & \pi_1(\mD^2)=1 \arrow{r}  & \pi_1(\mathcal{D}_{\pm}) \arrow{r}{{\mathcal{P}_{\pm}}_{*}} & \pi_1(\mathbf{K})\arrow{r} & 1.
                \end{tikzcd}
            \end{equation*}
            It follows that the kernel $\ker\big({i_{\pm}}_{*}\big)$ is isomorphic to $\mZ$, and thus, by Theorem \ref{thm: fund group of M}, $\pi_1(M)$ is isomorphic to $\pi_1(\mathbf{K})\big/N$, where $N$ is a cyclic normal subgroup. We denote 
            \[
                \pi_1(\mathbf{K})=\langle a,b\,|\, ab=ba^{-1}\rangle.
            \]
            A cyclic normal subgroup can only be $\{1\}$, $\langle a^k\rangle$ or $\langle b^{2k}\rangle$ for positive integer $k$. In every case, $\pi_1(M)$ is a semidirect product of two cyclic groups. This leads to a contradiction with either Lemma \ref{lemma: hyper non abelian} or Lemma \ref{lemma: hyper sub is Z} if $M$ is negatively curved. 
        \end{itemize}
    \end{enumerate}
    \paragraph{\bf Case 2: $d_{+}\le 2$ and $d_{-}=3$.} Since $\partial\mathcal{D}_{+}$ is connected, $\mathcal{P}_{-}:\mathcal{D}_{-}\to L_{-}$ must be a twisted $[-1,1]$-bundle. Thus $p_{-}:\partial\mathcal{D}_{-}\to L_{-}$ is a 2-sheeted covering, and we denote $\sigma:\partial\mathcal{D}_{-}\to\partial\mathcal{D}_{-}$ the covering transformation. Consider another smooth manifold $\tilde{M}$ admitting an LDDBD 
    \[
        \left(\mathcal{P}_{+}:\mathcal{D}_{+}\to L_{+},\mathcal{P}_{+}:\mathcal{D}_{+}\to L_{+},\phi^{-1}\circ\sigma\circ\phi:\partial\mathcal{D}_{+}\to \partial\mathcal{D}_{+}\right).
    \]
    It is straightforward to verify that $\tilde{M}$ is a 2-sheeted covering space of $M$. The incompatibility of $M$ with a negatively curved metric follows from that of $\tilde{M}$, which is discussed in Case 1.
    \par 
    \paragraph{\bf Case 3: $d_{+}=d_{-}=3$, and both disk bundles are trivial.} In this case, $\mathcal{D}_{\pm}$ is homeomorphic to $L_{\pm}\times[-1,1]$ and $\phi$ is a homeomorphism between $L_{+}\times\{-1\}\sqcup L_{+}\times\{1\}$ and $L_{-}\times\{-1\}\sqcup L_{-}\times\{1\}$. Without loss of generailty, we assume 
    \[
        \phi(x,\varepsilon)=\left\{
            \begin{aligned}
                & \big(\phi_1(x),-1\big)\in L_{-}\times\{-1\},&\ \text{if}\ \varepsilon=-1;\\
                & \big(\phi_2(x),1\big)\in L_{-}\times\{1\},&\ \text{if}\  \varepsilon=1.
            \end{aligned}
        \right.
    \]
    It follows that $L_{+}$ is homeomorphic to $L_{-}$, and $M$ is the total space of a $L_{+}$-bundle over the circle $\mS^1$, or equivalently, $M$ is homeomorphic to a mapping torus $\mathcal{M}_{\tau}$ of a homeomorphism $\tau$ of $L_{+}$, where $\tau=\phi_1^{-1}\circ\phi_2$. 
    \par 
    Firstly, we further assume $L_{+}$ is orientable and $\tau$ is orientation-preserving, which implies that $\mathcal{M}_{\tau}$ is also orientable. The homology spectral sequence([ref]) relates the Euler characteristics of the total space $\mathcal{M}_{\tau}$, base space $\mS^1$, and the fiber $L_{+}$. Denoting these Euler characteristics by $\chi\big(\mathcal{M}_{\tau}\big),\chi(L_{+})$, and $\chi(\mS^1)$, respectively, we have
    \[
        \chi\big(\mathcal{M}_{\tau}\big)=\chi(L_{+})\cdot\chi(\mS^1)=0.
    \]
    By the established relationship between the Euler characteristic and curvature in dimension 4 (see \cite{von1955curvature}), such a manifold $M$ can never admit a negatively curved metric. 
    \par  
    If $L_{+}$ is orientable but $\tau$ is orientation-reversing, we consider the mapping torus $\mathcal{M}_{\tau^2}$, which is clearly a 2-sheeted covering space of $\mathcal{M}_{\tau}$. Then, the incompatibility of $\mathcal{M}_{\tau}$ with a negatively curved metric follows from that of $\mathcal{M}_{\tau^2}$. 
    \par    
    If $L_{+}$ is non-orientable, let $\tilde{L}$ denote the orientable double cover of $L_{+}$. The homeomorphism $\tau$ of $L_{+}$ has two lifts to $\tilde{L}$, one of which is orientation-preserving and is denoted by $\tilde{\tau}$. The orientable mapping torus $\mathcal{M}_{\tilde{\tau}}$ is a 2-sheeted covering space of $\mathcal{M}_{\tau}$, and thus $\mathcal{M}_{\tau}$ cannot admit a negatively curved metric either. 
    \par  
    \paragraph{\bf Case 4: $d_{+}=d_{-}=3$, and both disk bundles are twisted.} This is the only remaining case, because whether $\partial\mathcal{D}_{+}$ and $\partial\mathcal{D}_{-}$ are both disconnected or both connected determines that these two $[-1,1]$-bundles must either both be trivial or both be twisted. In this case, $\mathcal{D}_{\pm}$ is homeomorphic to 
    \[
        \frac{[0,1]\times\partial\mathcal{D}_{\pm}}{(0,x)\sim \big(0,\sigma_{\pm}(x)\big)},
    \]
    where $\sigma_{\pm}$ is the covering transformation of the 2-sheeted covering map $p_{\pm}:\partial\mathcal{D}_{\pm}\to L_{\pm}$. From this viewpoint, $M$ is then homeomorphic to 
    \[
        \frac{[0,1]\times\partial\mathcal{D}_{+}}{(0,x)\sim (0,\sigma_{+}(x)),(1,x)\sim \big(1,\phi^{-1}\circ\sigma_{-}\circ\phi(x)\big)}. 
    \]
    Consider the mapping torus $\mathcal{M}_{\tau}$ with $\tau=\phi^{-1}\circ\sigma_{-}\circ\phi\circ\sigma_{+}$, which is homeomorphic to  
    \[
        \frac{[0,1]\times\partial\mathcal{D}_{+}}{\big(0,x\big)\sim \big(1,\phi^{-1}\circ\sigma_{-}\circ\phi\circ\sigma_{+}(x)\big)}.
    \]
    Immediately, we can define
    \begin{align*}
        \pi:\mathcal{M}_{\tau}&\to M\\
        (t,x)&\mapsto \left\{
            \begin{aligned}
                &\left(2t,x\right),\quad &\text{if}\ t\in\left[0,\frac{1}{2}\right];\\
                &\left(2-2t,\phi^{-1}\circ\sigma_{-}\circ\phi(x)\right),\quad &\text{if}\ t\in\left[\frac{1}{2},1\right].
            \end{aligned}
        \right.
    \end{align*}
    $\pi$ is a well-defined 2-sheeted covering map, because 
    \begin{align*}
        &\pi(0,x)=\pi(1,\phi^{-1}\circ\sigma_{-}\circ\phi(x))=(0,x)\\
        \sim&\big(0,\sigma_{+}(x)\big)=\pi\big(1,\phi^{-1}\circ\sigma_{-}\circ\phi\circ\sigma_{+}(x)\big)=\pi\big(0,\sigma_{+}(x)\big)
    \end{align*}
    for any $x\in\partial\mathcal{D}_{+}$.  Therefore, The incompatibility of $M$ with a negatively curved metric follows from that of $\mathcal{M}_{\tau}$, which is discussed in Case 3.
    \par  
    Combining the first two cases, we complete the proof of Theorem \ref{thm: fundamental group when singular}. Taking all four cases into account, we can immediately derive Theorem \ref{thm: exclu of hyperbolic}.

\section{Failure of Theorem \ref{thm: exclu of hyperbolic} in Dimension Three} \label{sec: fail in 3d}
    In proving the main theorem, we were able to obtain sufficient topological information in the four-dimensional case, allowing for a detailed discussion.
    As the dimension increases, however, the complexity of these decompositions grows rapidly, leading to significantly more intricate discussions. But surprisingly, Theorem \ref{thm: exclu of hyperbolic} fails to hold in the $3$-dimensional case. We can construct a counterexample as follows.
    \par    
    Suppose $S_g$ is an orientable closed surface of genus $g\ge 2$, and let $\tau$ be an orientation-preserving diffeomorphism on $S_g$. It is known that the mapping torus $\mathcal{M}_{\tau}$ admits a double disk bundle decomposition, where both base spaces are homeomorphic to $S_g$, and both bundles are trivial. By the Nielsen–Thurston classification (see, for example, \cite{thurston1986hyperbolic}), $\tau$ is isotopically periodic, reducible, or pseudo-Anosov, and when $\tau$ is pseudo-Anosov, the mapping torus $\mathcal{M}_{\tau}$ has a hyperbolic structure, that is, it admits a hyperbolic metric. This shows that the conclusion of Theorem \ref{thm: exclu of hyperbolic} does not hold for closed $3$-manifolds.



\bibliographystyle{amsplain}
\bibliography{Tran4D}

\end{document}